\documentclass[12pt]{elsarticle}

\usepackage{amssymb,amsmath,amsfonts,euscript,amscd,mathrsfs,amsthm,graphicx}
\usepackage[all]{xy}

\input xy
\xyoption{poly}

\DeclareMathOperator{\E}{\mathbb{E}}

\theoremstyle{definition}
\newtheorem{thm}{Theorem}[section]
\newtheorem{Cor}[thm]{Corollary}
\newtheorem{prop}[thm]{Proposition}
\newtheorem{lemma}[thm]{Lemma}

\newcommand{\lr}[1]{\langle #1 \rangle}

\begin{document}

\title{Symmetry of Endomorphism Algebras}
\author{Adam Allan}
\address{Dept. of Mathematics, St. Louis University, St. Louis, MO 63103}
\ead{aallan@slu.edu}


\begin{abstract}
Motivated by recent problems regarding the symmetry of Hecke algebras, we investigate the symmetry of the endomorphism algebra $E_P(M)$ for $P$ a $p$-group and $M$ a $kP$-module with $k$ a field of characteristic $p$. We provide a complete analysis for cyclic $p$-groups and the dihedral 2-groups. For the dihedral 2-groups, this requires the classification of the indecomposable modules in terms of string modules and band modules. We generalize our techniques to consider $E_{\Lambda}(M)$ for $\Lambda$ a Nakayama algebra, a local algebra, or even an arbitrary algebra.
\end{abstract}

\maketitle

\section{Background}

If $G$ is a finite group with subgroup $H$ and $k$ a field, then the endomorphism algebra $E_G(k_H\uparrow^G)$ of the permutation $kG$-module $k_H\uparrow^G$ is referred to as a Hecke algebra. It is not known in general when $E_G(k_H\uparrow^G)$ is symmetric or quasi-Frobenius, but this condition was explored in \cite{GreenSawada}, and proved useful in \cite{Naehrig} where $H$ was assumed to be a Sylow $p$-subgroup of $G$. Moreover, the centralizer algebras $kG^H$ for $H$ a subgroup of $G$ are Hecke algebras of the form $E_{H \times G}(k_{\Delta H}\uparrow^{H \times G})$ and have been recently studied in \cite{Ellers00}, \cite{Ellers04}, \cite{Ellers10}, and \cite{AAA}, and the author and J. Murray have been engaged in finding necessary and sufficient conditions guaranteeing symmetry of $kG^H$. For $H = G$ we have $kG^G = Z(kG)$ and in \cite{Muller} it was established that $Z(kG)$ is symmetric precisely when $G$ is $p$-nilpotent with abelian Sylow $p$-subgroups. It is natural, therefore, to analyze separately the case of $kG^H$ for $G$ a $p$-nilpotent group. This analysis, carried out in Theorem \ref{pNilpotent} below, led us to consider the separate problem of when $E_P(M)$ is symmetric for $P$ a $p$-group and $M$ a $kP$-module. It is this latter problem that we shall investigate more fully in this paper.

The paper will develop as follows. In section 2 we will review the necessary facts about symmetric algebras, analyze the symmetry of $kG^H$ for $G$ a $p$-nilpotent group, and briefly consider the symmetry of $E_P(M)$ for $P$ a cyclic $p$-group and $M$ a $kP$-module. In section 3, we analyze the symmetry of $E_P(M)$ for $P$ a dihedral 2-group, $k$ a field of characteristic 2, and $M$ an indecomposable $kP$-module. This will require the classification of indecomposable $kP$-modules in terms of band and string modules. A similar analysis is provided in section 4 for the Nakayama algebras, which generalizes our results for cyclic $p$-groups and has implications for the case of blocks with cyclic defect groups. Lastly, we provide some results of a general nature in section 5, with particular focus paid to local algebras. The results in this section neatly tie together some of the patterns observer earlier in the paper, while also pointing towards possible future research.

\textbf{Notation:} We denote finite groups by $G$, $H$, etc, $p$-groups by $P$, $Q$, and $R$, and modules by $M$ and $N$. Algebras are assumed to be finite dimensional over a ground field $k$, associative, and with identity; ${_{\Lambda}\text{mod}}$ denotes the finitely generated left $\Lambda$-modules; and our field $k$ is assumed to be algebraically closed of characteristic $p \geq 0$. A module is said to be \textit{isotypic} if all of its indecomposable direct summands are isomorphic, and the Loewy length $\ell\ell(M)$ of a module $M$ is defined to be the smallest $d$ for which $J^d(M) = 0$. Throughout this paper we write $E_{ij}$ for the matrix with zeroes everywhere except a 1 in the $i$th row and $j$th column; the size of $E_{ij}$ will be clear from context. Lastly, given matrices $A$ and $B$ of the appropriate size, $A \otimes B$ denotes their Kronecker product.

\section{Introductory Results}

Let ${\Lambda}$ be a $k$-algebra and recall that ${\Lambda}$ is quasi-Frobenius provided the left regular module ${_{\Lambda}{\Lambda}}$ is injective, ${\Lambda}$ is Frobenius provided ${_{\Lambda}{\Lambda}} \simeq {{\Lambda}_{\Lambda}}^* = \text{Hom}_k({\Lambda},k)$, ${\Lambda}$ is weakly symmetric provided ${\Lambda}$ is Frobenius and its Nakayama permutation is the identity, and ${\Lambda}$ is symmetric provided ${_{\Lambda}{\Lambda}_{\Lambda}} \simeq ({_{\Lambda}{\Lambda}_{\Lambda}})^*$. Evidently, each of these conditions is implied by the next and they are all left/right symmetric. If ${\Lambda}$ is quasi-Frobenius then $\text{Soc}({_{\Lambda}{\Lambda}}) = \text{Soc}({\Lambda}_{\Lambda})$ is a two-sided ideal; ${\Lambda}$ is Frobenius iff $\text{Soc}({_{\Lambda}{\Lambda}}) \simeq {\Lambda}/J({\Lambda})$ iff $\text{Soc}({\Lambda}_{\Lambda}) \simeq {\Lambda}/J({\Lambda})$; and if ${\Lambda}$ is local, then ${\Lambda}$ is quasi-Frobenius precisely when $\dim \text{Soc}({_{\Lambda}{\Lambda}}) = 1$, in which case ${\Lambda}$ is weakly symmetric. It is well-known that ${\Lambda}$ is Frobenius precisely when there is a linear map $\lambda : {\Lambda} \rightarrow k$ whose kernel contains no nonzero left (or equivalently right) ideals, and ${\Lambda}$ is moreover symmetric precisely when there exists such a $\lambda$ for which $\lambda(ab) = \lambda(ba)$ whenever $a,b \in {\Lambda}$. In this last case, we call $\lambda$ a symmetrizing form. If ${\Lambda}$ is symmetric with an idempotent $e$, then $e{\Lambda}e$ is symmetric as is $M_n({\Lambda})$ for any $n \geq 1$. So symmetry is preserved under Morita equivalences, as is the condition of being quasi-Frobenius. If $\Lambda_1$ and $\Lambda_2$ are algebras, then $\Lambda_1 \times \Lambda_2$ satisfies any of the four conditions precisely when $\Lambda_1$ and $\Lambda_2$ do so, $\Lambda_1 \otimes \Lambda_2$ satisfies any condition whenever both $\Lambda_1$ and $\Lambda_2$ do so, and $\Lambda_1$ and $\Lambda_2$ are (quasi-)Frobenius whenever $\Lambda_1 \otimes \Lambda_2$ is (quasi-)Frobenius by \cite{NakayamaII}. For convenience, we provide a proof to the following result.

\begin{lemma}\label{Monogenic}Any monogenic algebra ${\Lambda}$ is symmetric.\end{lemma}
\begin{proof}
We may write ${\Lambda} = k[X]/(p(X))$ for some nonzero polynomial $p(X)$. Since $k$ is algebraically closed, there are distinct $\alpha_1,\ldots,\alpha_n \in k$ with $p(X) = \prod (X-\alpha_i)^{e_i}$ for appropriate $e_i \geq 1$. By the Chinese Remainder Theorem ${\Lambda} \simeq \prod k[X]/((X-\alpha_i)^{e_i})$ and so without loss of generality we may assume $p(X) = (X-\alpha)^e$. However, the map $k[T]/(T^e) \rightarrow k[X]/((X-\alpha)^e)$ given by $T \mapsto X-\alpha$ is an isomorphism, and so we may suppose ${\Lambda} = k[T]/(T^e)$. Define $\lambda : {\Lambda} \rightarrow k$ by $\lambda(T^i) = \delta_{i,e-1}$ for $0 \leq i \leq e-1$, and suppose $\lambda({\Lambda}q(T)) = 0$ for some $q(T) \in {\Lambda}$. If $q(T) = \sum_{i=0}^{e-1} \beta_i T^i$, then $0 = \lambda(T^jq(T)) = \beta_{e-1-j}$ for $0 \leq j \leq e-1$. Hence, $q(T) = 0$ and $\lambda$ is a symmetrizing form.
\end{proof}

We also record, without proof, the next two lemmas for convenience.

\begin{lemma}\label{JCommute}Let $0 \not= \lambda \in k$ and write $J_m(\lambda)$ for the (upper triangular) irreducible Jordan block with eigenvalue $\lambda$. Then $\mathcal{C} := C_{M_m(k)}(J_m(\lambda))$ consists of all upper triangle matrices that are constant along diagonals. In particular, $\mathcal{C} \simeq k[T]/(T^m)$ as algebras, and so $\mathcal{C}$ is symmetric.\end{lemma}

\begin{lemma}\label{fixed}Suppose $M$ is a $kG$-module with $k$-subspaces $M_1, \ldots, M_e$ such that $M = \bigoplus M_i$ and $G$ permutes the sets $\{ M_i \}$ transitively amongst themselves. Write $H$ for the set-wise stabilizer of $M_1$. Then $\text{Tr} : M^H \rightarrow M^G$ given by $\text{Tr}(x) = \sum_{g \in G/H} {^g x}$ is a linear isomorphism.\end{lemma}

We now provide the result that has motivated much of the work in this paper.

\begin{thm}\label{pNilpotent}Suppose $k$ is a field of characteristic $p$ and $G$ is a $p$-nilpotent group with $N = O_{p'}(G)$ and $P \in \text{Syl}_p(G)$. Then $kG^P$ is Frobenius if and only if $P$ is abelian and $E_{T_P(M)}(M)$ is Frobenius for every $M \in \text{Irr}(kN)$. \end{thm}

Note: if $M \in \text{Irr}(kN)$ then there is a unique irreducible $T_G(M)$-module, also denoted by $M$, whose restriction to $N$ equals $M$, and so $E_{T_P(M)}(M)$ is sensibly defined.

\begin{proof}
Suppose for the moment that $kG^P$ is Frobenius. Write $\bar{G} = G/N$ and let $e = \frac{1}{|N|}\sum_{n \in N} n$ so that $b_0 = ekG$ is the principal block of $kG$. Then $b_0 \simeq k\bar{G}$ and $b_0^P \simeq k\bar{G}^{\bar{P}}$ where $\bar{P}$ acts on $k\bar{G}$ by conjugation. Since $\bar{G}$ is a $p$-group for which $k\bar{G}^{\bar{P}}$ is Frobenius, we know from \cite{AAA} that $\bar{P} \leq Z(\bar{G})$, and hence $P$ is abelian.

Under the condition that $P$ is abelian, we now derive necessary and sufficient conditions for $kG^P$ to be Frobenius. Suppose $b$ is a block of $kN$ with corresponding block idempotent $e_b$. Let $\{e_b = e_{b_1},\ldots,e_{b_r} \}$ be the conjugates of $e_b$ under the action of $P$. Then there is a unique block $B$ of $kG$ that covers $b$, and in fact $e_B = e_{b_1} + \cdots + e_{b_r}$. So $B = e_B kG = \bigoplus e_{b_i} kG = \bigoplus_i b_i kP$. Note that $P$ permutes the subspaces $\{ b_i kP \}$ transitively by conjugation. Also, $Q := C_P(b)$ consists of all elements in $P$ that stabilize $b kP$ set-wise. Therefore, $\text{Tr} : (b kP)^Q \rightarrow B^P$ is a linear isomorphism by Lemma \ref{fixed}.

By consideration of supports, $b kP = \bigoplus_{p \in P} b p$. Moreover, $b p$ is a $Q$-invariant subspace since $P$ is abelian. In particular, $(b kP)^Q = \bigoplus_{p \in P} b^Q p$ and $\text{Tr}(xp) = \text{Tr}(x)p$ for $x \in b^Q$. There is a well-defined linear map $b^Q \otimes kP \rightarrow B^P$ given by $x \otimes y \mapsto \text{Tr}(x)y$. Suppose $\xi$ lies in the kernel of this map and write $\xi = \sum x_p \otimes p$ for some $x_p \in b^Q$. Then $0 = \sum \text{Tr}(x_p)p = \text{Tr}(\sum x_p p)$ and $\sum x_p p \in (bkP)^Q$. So $\sum x_p p = 0$ and hence $x_p = 0$ for all $p$, so that $\xi = 0$. It follows from the previous remarks that the map $b^Q \otimes kP \rightarrow B^P$ is surjective, and we claim that it is an algebra homomorphism. Since $\text{Tr}(x)p = p\text{Tr}(x)$ for $x \in b^Q$ it suffices to show that $\text{Tr}(xy) = \text{Tr}(x)\text{Tr}(y)$ for $x,y \in b^Q$. If $t,t' \in P$ with $tQ \not= t'Q$ then $txt^{-1}$ and $t'yt'^{-1}$ lie in different sets from $\{b_1,\ldots,b_r\}$ and so $txt^{-1}t'yt'^{-1} = 0$. Consequently, $\text{Tr}(x)\text{Tr}(y) = \sum_{t \in P/Q} txyt^{-1} = \text{Tr}(xy)$ as required.

Thus, $B^P \simeq b^Q \otimes kP$ as algebras. Note that $B^P$ is Frobenius precisely when $b^Q$ is Frobenius since $kP$ is always Frobenius. However, $b$ is a block of $kN$ of defect zero since $N$ is a $p'$-group, and so $b = E_k(M_b)$ for some $M_b \in \text{Irr}(N)$. By Clifford Theory there is a unique irreducible $T_G(M_b)$-module whose restriction to $N$ equals $M_b$. We also write this $T_G(M_b)$-module as $M_b$, and note that $b^Q \simeq E_Q(M_b)$. This establishes the result.
\end{proof}

Note that the splitting $B^P \simeq b^Q \otimes kP$ is similar to the main result from \cite{KulshammerSplitting}, though that result is obtained under slightly different assumptions. The result should also be true with the symmetric condition replacing the Frobenius condition, but we were unable to find a reference that symmetry is preserved under taking tensor summands. We now turn our attention to the analysis of when $E_P(M)$ is a Frobenius or symmetric algebra, as per the condition in Proposition \ref{pNilpotent}. We begin with the simplest case of a cyclic $p$-group.

\begin{prop}\label{CyclicModules}Suppose $P$ is a cyclic $p$-group and $M \in {_{kP}\text{mod}}$. Then $E_P(M)$ is quasi-Frobenius iff $M$ is isotypic, in which case $E_P(M)$ is symmetric.\end{prop}
\begin{proof}
By Theorem \ref{localAlgebra} to be proved later, it suffices to show that $E_P(M)$ is symmetric whenever $M$ is indecomposable. For this, let $x$ be a generator of $P$ of order $q$ and choose a basis $\{ v_1,\ldots,v_r \}$ of $M$ with $r \leq |P|$ such that $x$ acts via the matrix $J_m(1)$ with respect to this basis. Then $E_P(M) = \mathcal{C}$ is symmetric by Lemma \ref{JCommute}.
\end{proof}

\section{Dihedral 2-groups}

We now turn to one of the most interesting situations where all indecomposable $kP$-modules are known: dihedral 2-groups in characteristic 2. Begin by letting $D_{\infty} = \lr{x,y | x^2 = y^2 = 1}$ be the infinite dihedral group and $k$ an algebraically closed field of characteristic 2. Write $\Lambda = kD_{\infty}$ and define generators of $\Lambda$ by $X = x-1$ and $Y = y-1$. If $D_{4q}$ is the dihedral group of order $4q$ for $q \geq 1$ a power of 2, then $kD_{4q} \simeq \Lambda/I_q$ where $I_q = \lr{(XY)^q-(YX)^q}$. So an indecomposable $kD_{4q}$-module is the same as an indecomposable $\Lambda$-module $M$ with $I_q \subseteq \text{Ann}_{\Lambda}(M)$, and moreover $E_{D_{4q}}(M) \simeq E_{\Lambda}(M)$. So we may concentrate on the totality of all $\Lambda$-modules for the moment. These fall into one of two types: string modules and band modules.

We first construct the string modules, modifying the treatment from \cite{BensonI} only slightly. More precisely, we let $\mathcal{W}$ denote the words (= strings) of finite length on the letters $a, b, a^{-1}$, and $b^{-1}$ with the caveat that any appearance of $a$ is followed by $b$ or $b^{-1}$, and any appearance of $b$ is followed by $a$ or $a^{-1}$. We also include words $1_a$ and $1_b$ of length zero in $\mathcal{W}$. For every $w \in \mathcal{W}$ we write $w = l_1 \cdots l_n$ for $l_i \in \{a,b,a^{-1},b^{-1}\}$, and we define $n = |w|$ as the length of $w$. Let $\mathcal{B} = \{z_1,\ldots,z_{n+1}\}$ be the basis of an $(n+1)$-dimensional $k$-space $M_w$. We endow $M_w$ with a $\Lambda$-module structure of as follows:

\begin{equation}\label{string}
X.z_i = \begin{cases}
z_{i+1} &\text{if $l_i = a$} \\
z_{i-1} &\text{if $l_{i-1} = a^{-1}$} \\
0 &\text{otherwise}
\end{cases}
\end{equation}

and similarly for $Y$ with $a$ replaced by $b$. For $w = 1_a, 1_b$ we simply have $M_w = k$ with $X$ and $Y$ acting as zero. It is convenient to visualize $M_w$ in terms of diagrams. For instance, we assign to the word $a^{-1}bab^{-1}a^{-1}$ the diagram

\begin{figure}[here]
$$\xymatrix{z_1 & z_2 \ar[l]_a \ar[r]^b & z_3 \ar[r]^a & z_4 & z_5 \ar[l]_b & z_6 \ar[l]_a }$$
\label{fig5.1}
\end{figure}

Then relative to the basis $\mathcal{B}$ we have $X = E_{12}+E_{43}+E_{56}$ and $Y = E_{32}+E_{45}$ where $E_{ij} : M_w \rightarrow M_w$ by $E_{ij}(z_k) = \delta_{jk}z_i$. In the previous diagram, we say that $z_1$ and $z_4$ are 'sinks', while $z_2$ and $z_6$ are 'sources'. If $w \in \mathcal{W}$ then we may form $w^{-1} \in \mathcal{W}$ with the understanding that $1_a^{-1} = 1_b$ and $1_b^{-1} = 1_a$. It is straightforward to show that $M_{w^{-1}} \simeq M_w$. With this notation we may establish the following.

\begin{thm}\label{StringModules}Let $D_{\infty}$ be the infinite dihedral 2-group, $k$ an algebraically closed field of characteristic 2, and $\Lambda = kD_{\infty}$. If $M_w$ is a string module for $\Lambda$, then $E_{\Lambda}(M_w)$ is a quasi-Frobenius algebra precisely when $w$ or $w^{-1}$ equals one of $1_a, a, b, (ba)^l$, or $(ab)^l$ for some $l \geq 1$. Moreover, $E_{\Lambda}(M_w)$ is symmetric whenever it is quasi-Frobenius.\end{thm}
\begin{proof}
If $w = 1_a$ or $w = 1_b$ then $M_w = k$ and $E_{\Lambda}(k) \simeq k$ is symmetric. So assume $|w| \geq 1$. We regard $\E = E_{\Lambda}(M_w)$ as the collection of matrices that commute with $X$ and $Y$, where $X$ and $Y$ are regarded as matrices via their action on $M_w$ with respect to the basis $\mathcal{B}$. If we switch $a$ with $b$ and $a^{-1}$ with $b^{-1}$ in $w$, then the effect is to switch $X$ with $Y$, and $\E$ remains unchanged. Moreover, if we switch $a$ with $a^{-1}$ and $b$ with $b^{-1}$, then the effect is to switch $X$ and $Y$ with their transposes $X^t$ and $Y^t$, respectively. The centralizer in $M_n(k)$ of $X^t$ and $Y^t$ equals $\E^t$, and $\E^t \simeq \E^{\text{op}}$ is quasi-Frobenius (or symmetric) precisely when $\E$ is the same. So we may assume that $w$ begins with $a$.

Note that $XE_{ij} = E_{kj}$ if $X$ sends $z_i$ to $z_k$, and $XE_{ij} = 0$ if $X(z_i) = 0$. So if $z_i$ is a sink then $XE_{ij} = 0$, and if $z_j$ is a source then $E_{ij}X = 0$; similarly for $Y$. Hence, $E_{ij} \in \mathbb{E}$ whenever $z_i$ is a sink and $z_j$ is a source. Suppose $T \in \mathbb{E}$ and write $T = \sum t_{lk}E_{lk}$ for some scalars $t_{lk}$, so that $TE_{ij} = \sum_l t_{li}E_{lj}$. Then

$$0 = TXE_{ij} = XTE_{ij} = \sum_l t_{li} XE_{lj} = \sum_{\forall l \exists k (X(z_l) = z_k)} t_{li} E_{kj}$$

Since $X$ never maps two basis elements to the same basis element, we have by linear independence that $t_{li} = 0$ whenever $X(z_l) \not= 0$. Similarly, $t_{li} = 0$ whenever $Y(z_l) \not= 0$. Hence, $t_{li} = 0$ unless $z_l$ is a sink. Thus, if $z_i$ is a sink and $z_j$ a source then

$$TE_{ij} = \sum_{l : \text{$z_l$ is a sink}} t_{li} E_{lj}$$

Therefore, we obtain a submodule $N_j$ of ${_{\E} \E}$ by fixing a source $z_j$, and letting $N_j$ be spanned by all $E_{lj}$ for $z_l$ an arbitrary sink. Clearly $N_j \cap N_{j'} = 0$ for $j \not= j'$, and ${_{\E} \E}$ is an indecomposable projective since $\mathbb{E}$ is a local algebra. In particular, if the diagram has at least two sources $z_j$ and $z_{j'}$, then $\mathbb{E}$ is not quasi-Frobenius since $\text{Soc}({_{\E} \E}) \supseteq \text{Soc}(N_j) \oplus \text{Soc}(N_{j'})$ is at least 2-dimensional.

We assume then that there is a unique source, which must be $z_1$ since $w$ starts with $a$. Hence, $w$ is $a$, $(ab)^l$, or $(ab)^la$ for some $l \geq 1$. We compute a basis for $\mathbb{E}$ explicitly in these cases. Let $T \in \mathbb{E}$ and write $T(z_j) = \sum_i t_{ij} z_i$ for $z_i \in \mathcal{B}$. Then

$$\sum_{i : l_i = a} t_{ij} z_{i+1} = \sum_i t_{ij}Xz_i = XT(z_j) = TX(z_j) = \delta_{\texttt{true},l_j = a} \sum_i t_{i+1,j+1} z_{i+1}$$

and similarly for $YT = TY$. In particular $T \in \mathbb{E}$ precisely when it obeys

\begin{equation*}
\begin{split}
\text{$l_i = a$ and $l_j \not= a$} &\rightarrow t_{ij} = 0 \\
\text{$l_i = b$ and $l_j \not= b$} &\rightarrow t_{ij} = 0 \\
(l_i,l_j) \in \{(a,a),(b,b)\} &\rightarrow t_{ij} = t_{i+1,j+1} \\
\text{$l_n = b$ and $l_j = a$} &\rightarrow t_{n+1,j+1} = 0 \\
\text{$l_n = a$ and $l_j = b$} &\rightarrow t_{n+1,j+1} = 0 \\
\end{split}
\end{equation*}

Therefore, $(t_{ij})$ must be a lower triangular matrix that is constant along diagonals and satisfies $t_{i1} = 0$ unless $i$ is odd or $i = n+1$. For $i \geq 0$ define

$$T_i = \sum_{j=1}^{n-2i+1} E_{j+2i,j}$$

where an empty summation is understood to equal 0. Observe that $T_0 = I$ and $T_iT_{i'} = T_{i+i'}$. If $w = (ab)^l$ then $\{T_0,T_1,\ldots,T_l\}$ is a basis of $\mathbb{E}$. Since $T_1^j = T_j$ for $1 \leq j \leq l$ we see that $\mathbb{E}$ is symmetric by Lemma \ref{Monogenic}. On the other hand, if $w = (ab)^la$ then $\mathcal{B}' = \{ T_0,T_1,\ldots,T_l,E_{2l+1,1} \}$ is a basis of $\mathbb{E}$. When $l = 0$ (i.e. $w = a$) we have $\mathbb{E}$ is symmetric by Lemma \ref{Monogenic}, and so we assume $l \geq 1$. Then the radical of $\mathbb{E}$ is spanned by $\mathcal{B}' \setminus \{T_0\}$, which annihilates $T_l$ and $E_{2l+1,1}$ by left multiplication, and thus shows that $\dim \text{Soc}({_{\E} \E}) \geq 2$. So $\mathbb{E}$ is not quasi-Frobenius, thus completing the proof.
\end{proof}

Now we deal with the band modules, where the details will not be too different from the details for string modules. To begin, the $n$th power of a word of even length is simply the juxtaposition of the word $n$ times. We call a string $w \in \mathcal{W}$ a \textit{band} if $w$ has positive even length and $w$ is not a power of a proper even length subword. Given a band $w$, an integer $m \geq 1$, and a scalar $0 \not= \lambda \in k$, we let $V$ be the $k$-space of dimension $m$ and write $M(w,m,\lambda) = \bigoplus_{i=1}^n Vz_i$. Moreover, the $\Lambda$-module structure on $M(w,m,\lambda)$ is similar to the $\Lambda$-module structure on $M(w)$, and it is easier simply to provide an example in place of a detailed description (though the reader is referred to \cite{BensonI} for the latter). First, we let $J_m(\lambda)$ be the irreducible Jordan block with minimal polynomial $p(X) = (X-\lambda)^m$. Then given $w = aba^{-1}b^{-1}$ we have the diagram

\begin{figure}[here]
$$\xymatrix{z_1 \ar[r]^a \ar@/_2pc/[rrr]^b & z_2 \ar[r]^b & z_3 & z_4 \ar[l]_a}$$
\label{fig5.2}
\end{figure}

Fixing a basis $\mathcal{B}$ of $V$ we take the basis $\mathcal{A} = \coprod_{i=1}^n \mathcal{B}z_i$ of $M(w,m,\lambda)$. With respect to $\mathcal{A}$ we have that the actions of $X$ and $Y$ are given by

\begin{equation*}
X =
\begin{pmatrix}
0 & 0 & 0 & 0 \\
I & 0 & 0 & 0 \\
0 & 0 & 0 & I \\
0 & 0 & 0 & 0
\end{pmatrix}
\hspace{20 pt}
Y =
\begin{pmatrix}
0 & 0 & 0 & 0 \\
0 & 0 & 0 & 0 \\
0 & I & 0 & 0 \\
J_m(\lambda) & 0 & 0 & 0
\end{pmatrix}
\end{equation*}

where each entry in $X$ and $Y$ is an $m \times m$ matrix and $I$ represents the identity. So the matrix 'form' of $X$ and $Y$ is determined by the diagram, as it was for string modules. It is necessary to modify our earlier notation by writing, for example, $X(z_1^*) = z_2^*$, $X(z_2^*) = 0$, and $Y(z_1^*) = z_4^*$, where $z_i^*$ is the subspace of $M(w,m,\lambda)$ with basis $\mathcal{B}z_i$. We also write, for example, $\text{Ker}(Y) = \{ z_3^*, z_4^* \}$. In the diagram above we continue to call $z_3^*$ a sink and $z_1^*$ a source. There are also isomorphisms of band modules given by $M(w,m,\lambda) \simeq M(w^{-1},m,\lambda^{-1})$ and $M(w,m,\lambda) \simeq M(w',m,\lambda)$ if $w'$ is obtained by cyclically permuting the letters of $w$.  It is also convenient to define an equivalence relation on $\mathcal{W}$ by identifying $w$ with $w'$ and any cyclic permutation of $w$. We now characterize the band modules with symmetric endomorphism algebras.

\begin{thm}\label{BandModules}Suppose $M(w,m,\lambda)$ is a band module and define $\mathbb{E} = E_{\Lambda}(M(w,m,\lambda))$. Then $\mathbb{E}$ is quasi-Frobenius precisely when (1) $w$ is equivalent to $ab$, or (2) $m = 1$ and $w$ is equivalent to a word in

$$\{ ab^{-1},aba^{-1}b^{-1},abab^{-1}a^{-1}b^{-1},\ldots \}$$

Whenever $\E$ is quasi-Frobenius, it is weakly symmetric. Moreover, $\mathbb{E}$ is symmetric only in case (1), in case (2) for $|w| = 2$, or in case (2) when $|w| \equiv_4 0$ and $\lambda = 1$.\end{thm}
\begin{proof}
Let $n = |w| > 0$. We write $E_{ij}$ for the $n \times n$ matrix with zeroes everywhere except for a 1 in the $(i,j)$th entry, and given an $m \times m$ matrix $A$ we write $E_{ij} \otimes A$ for the $nm \times nm$ matrix formed by taking the Kronecker product of $E_{ij}$ and $A$. If $z_i^*$ is a sink, $z_j^*$ is a source, and $A$ is an arbitrary $m \times m$ matrix, then by the same argument as in Theorem \ref{StringModules} we see that $X$ and $Y$ annihilate $E_{ij} \otimes A$ on the left and right, and hence $E_{ij} \otimes A \in \mathbb{E}$. Given $T \in \mathbb{E}$ we can write $T = \sum E_{lk} \otimes A_{lk}$ for some $m \times m$ matrices $A_{lk}$. Then

$$0 = TX(E_{ij} \otimes I) = XT(E_{ij} \otimes I) = \sum_l X(E_{lj} \otimes A_{li}) = \sum_{\forall l \exists k (X(z_l^*) = z_k^*)} E_{kj} \otimes A_{li}$$

Since there is at most one $z_l^*$ for which $X(z_l^*) = z_k^*$, we see that $A_{li} = 0$ whenever $X(z_l^*) \not= 0$. Similarly, $A_{li} = 0$ whenever $Y(z_l^*) \not= 0$ since $J_m(\lambda)$ is an invertible matrix. Thus, if $A \in M_m(k)$, $z_i^*$ is a sink, and $z_j^*$ is a source, then

\begin{equation}\label{bandT}
T(E_{ij} \otimes A) = T(E_{ij} \otimes I)(I \otimes A) = \sum_{l : \text{$z_l^*$ is a sink}} E_{lj} \otimes A_{li}A
\end{equation}

In other words, we obtain a submodule $N_j$ of ${_{\E} \E}$ by fixing a source $z_j^*$, and letting $N_j$ be spanned by all $E_{lj} \otimes A$ where $A \in M_m(k)$ is arbitrary and $z_l^*$ is an arbitrary sink. Clearly $N_j \cap N_{j'} = 0$ for $j \not= j'$, and ${_{\E} \E}$ is an indecomposable projective since $\mathbb{E}$ is a local algebra. In particular, if the diagram has at least two sources then $\mathbb{E}$ is not quasi-Frobenius since $\text{Soc}({_{\E} \E})$ is at least 2-dimensional.

If there is no source, then the diagram is cyclic and hence $|w| = 2$ since $w$ is not a power of any of its proper subwords. So $w$ is equivalent to $a^{-1}b^{-1}$, and we have

\begin{equation*}
X =
\begin{pmatrix}
0 & I \\
0  & 0
\end{pmatrix}
\hspace{20 pt}
Y =
\begin{pmatrix}
0 & 0 \\
J_m(\lambda) & 0
\end{pmatrix}
\end{equation*}

It is easy to see that $\mathbb{E}$ consists of all matrices of the form $I \otimes A$ where $I$ is the $2 \times 2$ identity matrix and $A \in \mathcal{C}$, as per the notation of Lemma \ref{JCommute}. In particular, $\mathbb{E}$ is a symmetric algebra.

So we assume that the diagram has one source, in which case it must also have a sink. If there are two sinks, then there must be at least two sources, and hence there is precisely one source and one sink. Because $w$ may be replaced with its cyclic permutations, we may assume that $z_1^*$ is a source. So if $w = l_1 \cdots l_n$ then either $(l_1,l_n) = (a,b^{-1})$ or $(l_1,l_n) = (b,a^{-1})$. However, switching $a$ with $b$ and $a^{-1}$ with $b^{-1}$ has the effect of switching $X$ with $Y$, which leaves $\mathbb{E}$ unaffected. So we may assume $l_1 = a$ and $l_n = b^{-1}$, and we also write $z_l^*$ for the sink. We first show that $m = 1$ if $\mathbb{E}$ is quasi-Frobenius. To this end, we claim that if $T = \sum E_{ij} \otimes A_{ij}$ then $A_{ii} = A_{11}$ for all $i$ and $A_{11} \in \mathcal{C}$. If $X(z_r^*) = z_s^*$ then $X(E_{rs} \otimes I) = E_{ss} \otimes I$ and hence

\begin{equation*}
\begin{split}
\sum_{i'} E_{i's} \otimes A_{i's} &= T(E_{ss} \otimes I) = TX(E_{rs} \otimes I) \\
&= XT(E_{rs} \otimes I) = X\sum_i E_{is} \otimes A_{ir} = \sum_{\forall i \exists k (X(z_i^*) = z_k^*)} E_{ks} \otimes A_{ir}\\
\end{split}
\end{equation*}

Taking $i' = s$ on the left-hand side and $i = r$ on the right hand side yields $E_{ss} \otimes A_{ss} = E_{ss} \otimes A_{rr}$ and so $A_{rr} = A_{ss}$. If $z_r^* \not= z_1^*$ and $Y(z_r^*) = z_s^*$, then $Y(E_{rs} \otimes I) = E_{ss} \otimes I$ and hence $A_{rr} = A_{ss}$, by the same argument. When $z_r^* = z_1^*$ we obtain $Y(z_1^*) = z_n^*$ and $Y(E_{1n} \otimes I) = E_{nn} \otimes J_m(\lambda)$. So analogous computations yield

$$\sum_{i'} E_{i'n} \otimes A_{i'n}J_m(\lambda) = E_{nn} \otimes J_m(\lambda)A_{11} + \sum_{\forall i \not= 1 \exists k (Y(z_i^*) = z_k^*)} E_{kn} \otimes A_{i1}$$

In particular, by taking $i' = n$ we see that $A_{nn}J_m(\lambda) = J_m(\lambda)A_{11}$. So $A_{ii} = A_{11}$ for all $i$ and $A_{11} \in \mathcal{C}$. Notice that (\ref{bandT}) becomes

$$T(E_{l1} \otimes A) = E_{l1} \otimes A_{ll}A = E_{l1} \otimes A_{11}A$$

since $z_l^*$ is the unique sink. Because $\mathcal{C}$ consists of upper triangular matrices, for $E_{1i} \in M_m(k)$ we see that $A_{11}E_{1i}$ is a multiple of $E_{1i}$. Hence, the $k$-linear span $S_i$ of $E_{l1} \otimes E_{1i}$ is a 1-dimensional submodule of ${_{\E} \E}$, and so $\bigoplus_{i=1}^m S_i \subseteq \text{Soc}({_{\E} \E})$. In particular, if $\mathbb{E}$ is quasi-Frobenius then $1 = \dim \text{Soc}({_{\E} \E}) \geq \dim \bigoplus_{i=1}^m S_i = m$ and hence $m = 1$.

So we assume $m = 1$, and continue our convention that $z_1$ is a source and $z_l$ is a sink. Since $n = |w|$ is a positive even integer, we can write $n = 2p$. There are two cases to consider: $z_l \not= z_{p+1}$ and $z_l = z_{p+1}$. If $z_l \not= z_{p+1}$ then $n > 2$, and $2 \leq l \leq p$ or $p+1 < l \leq n$. If $p+1 < l \leq n$ then replacing $w$ with $w^{-1}$ results in a diagram with $z_1$ as source and $z_{2(p+1)-l}$ as sink, where $2 \leq 2(p+1)-l \leq p$. So we suppose $2 \leq l \leq p$, and we may assume the following situation

\begin{figure}[here]
$$\xymatrix{z_1 \ar@/_2pc/[rrrrrrr]_b \ar[r]^a & \cdots \ar[r] & z_{l-1} \ar[r] & z_l & z_{l+1} \ar[l] & z_{l+2} \ar[l] & \ar[l] \cdots & \ar[l] z_n}$$
\label{fig5.3}
\end{figure}

where $l+2, 2l \leq n$ since $p > 1$. It is clear that if $Z \in \{X,Y\}$ then $z_i \in \text{Ker}(Z)$ precisely when $z_i^* = Z(z_j^*)$ for some $z_j^*$. Now as in the proof to Theorem \ref{StringModules} we see that $T \in \mathbb{E}$ precisely when it obeys the two rules:

\begin{equation}\label{rules}
\begin{split}
\text{$z_i \not\in \text{Ker}(Z)$ and $z_s \in \text{Ker}(Z)$} &\rightarrow t_{is} = 0 \\
\text{$\mu_1 z_i = Z(z_j)$ and $\mu_2 z_s = Z(z_r)$ for $\mu_i \in \{1,\lambda\}$} &\rightarrow \mu_2 t_{is} = \mu_1 t_{jr} \\
\end{split}
\end{equation}

Define an element $\zeta$ of $M_n(k)$ by

$$\zeta = \begin{cases}
E_{l+1,1} + \lambda^{-1}E_{ln} & \text{if $Y(z_{l+1}) = z_l$}\\
E_{l+1,1} + E_{l2} & \text{if $X(z_{l+1}) = z_l$}\\
\end{cases}$$

We suppose that $Y(z_{l+1}) = z_l$; the argument for $X(z_{l+1}) = z_l$ is similar and in fact simpler. Note that $X\zeta = 0$ since $z_{l+1}, z_l \in \text{Ker}(X)$, and $\zeta X = 0$ since $z_1, z_n \not\in \text{Ker}(X)$. Moreover, $Y\zeta = E_{l,1}$ since $Y(z_{l+1}) = z_l$ and $Y(z_l) = 0$, and $\zeta Y = E_{l,1}$ since $Y(z_1) = \lambda z_n$ and $z_1 \not\in \text{Ker}(Y)$. In particular, $\zeta \in \mathbb{E}$. We claim that $k\zeta$ is a submodule of ${_{\E} \E}$. So suppose $T \in \mathbb{E}$, write $T = \sum t_{ij}E_{ij}$ for scalars $t_{ij}$, and let $Z \in \{X,Y\}$. Observe that $T\zeta = \sum_i t_{i,l+1}E_{i1} + \lambda^{-1}t_{il}E_{in}$. By the first rule, $t_{il} = 0$ for $i \not= l$ and $t_{i,l+1} = 0$ whenever $X(z_i) \not= 0$; and by the second rule, $t_{i,l+1} = \mu t_{jl}$ for some $\mu \not= 0$ whenever $Y(z_i^*) = z_j^*$. So $t_{i,l+1} = 0$ for $i \not= l,l+1$ and $t_{l+1,l+1} = t_{ll}$. Thus, $T\zeta = t_{l,l+1}E_{l1} + t_{ll}\zeta$ and it suffices to show that $t_{l,l+1} = 0$ for all $T \in \mathbb{E}$. Because $Y(z_{l+1}) = z_l$, induction on the second rule yields $t_{l,l+1} = t_{1,2l}$. Since $z_{2l}$ is not the sink and $z_1$ is the source, $t_{1,2l} = 0$ by the first rule. Therefore, $k\zeta$ is a nonzero (simple) submodule of ${_{\E} \E}$, and $k\zeta \oplus kE_{l1} \subseteq \text{Soc}({_{\E} \E})$ so that $\dim \text{Soc}({_{\E} \E}) \geq 2$ and hence $\mathbb{E}$ is not quasi-Frobenius.

Lastly, we assume $z_l = z_{p+1}$ and show that $\mathbb{E}$ is a Frobenius algebra. Graphically, this means that the two 'paths' in the diagram between the source and the sink have the same length and these two paths are the unique paths in the diagram of maximal length. If $n = 2$ then $\mathbb{E}$ is 2-dimensional and hence symmetric by Lemma \ref{Monogenic}. So for convenience, we assume $n > 2$. Define a relation $\mathcal{R}$ on $I = \{(i,j) : 1 \leq i,j \leq n\}$ by $(j,r)\mathcal{R}(i,s)$ if $Z(z_j^*) = z_i^*$ and $Z(z_r^*) = z_s^*$ for some $Z \in \{X,Y\}$. We also write $Z : (j,r) \rightarrow (i,s)$. Let $\sim$ denote the equivalence relation on $I$ induced by $\mathcal{R}$. Write $I_0$ for the set of all $(i,s)$ with $z_i^* \not\in \text{Ker}(Z)$ and $z_s^* \in \text{Ker}(Z)$ for some $Z \in \{X,Y\}$, and define $I_0^*$ as all $(i,s)$ with $(i,s) \sim (i',s')$ for some $(i',s') \in I_0$, so that $I_0^*$ is a union of equivalence classes. For $(i,s) \in I$ write $[i,s]$ for the equivalence class containing $(i,s)$. We will show that $\mathbb{E}$ has a basis indexed by $\{ [i,1] : 1 \leq i \leq n \}$. To do this, we first show that $\{ (i,1) : 1 \leq i \leq n \}$ are distinct representatives for the classes $I \setminus I_0^*$.

To begin, if $(j,r)\mathcal{R}(i,s)$ and $(j,r)\mathcal{R}(i',s')$ with $(i,s) \not= (i',s')$, then there are $Z_1, Z_2 \in \{X,Y\}$ with $Z_1 \not= Z_2$, $Z_1 : (j,r) \rightarrow (i,s)$, and $Z_2 : (j,r) \rightarrow (i',s')$. But then $z_j = z_r = z_1$ since $z_1$ is the unique source. Suppose $(j,r)\mathcal{R}(i,s)$ and $(j',r')\mathcal{R}(i,s)$ with $(j,r) \not= (j',r')$. Again, there are $Z_1, Z_2 \in \{X,Y\}$ with $Z_1 \not= Z_2$ and $Z_1 : (j,r) \rightarrow (i,s)$ and $Z_2 : (j',r') \rightarrow (i,s)$. From $\{ Z_1(z_k^*) \} \cap \{ Z_2(z_k^*) \} = \{ z_{p+1}^* \}$ we get $z_i = z_s = z_{p+1}$. Now it is easy to see that $[p+1,p+1] = \{(i,i) : 1 \leq i \leq n\}$. Therefore, if $(i,s) \in I$ with $i \not= s$, then the elements in $[i,s]$ are linearly ordered by $\mathcal{R}$. That is, there is a (unique) maximal chain

$$\xymatrix{(j_1,r_1)\mathcal{R}(j_2,r_2) & \cdots & (j_{t-1},r_{t-1})\mathcal{R}(j_t,r_t)}$$

with $(j_u,r_u) = (i,s)$ for some $u$, in which case $[i,s] = \{ (j_v,r_v) : 1 \leq v \leq t \}$. Since $z_1$ is a source, $(i,1)$ for $i \not= 1$ is the left-most element in any such chain in which it appears. Let the above denote a maximal chain for $(i,1)$ with $i \not= 1$, and suppose $(j_v,r_v) \in I_0$ for some $1 \leq v \leq t$. Clearly $v > 1$, and so there is $Z_1 : (j_{v-1},r_{v-1}) \rightarrow (j_v,r_v)$. Let $Z_2 \in \{X,Y\} \setminus \{Z_1\}$ and note that $j_v \in \text{Ker}(Z_1)$. So we must have $j_v \not\in \text{Ker}(Z_2)$ and $r_v \in \text{Ker}(Z_1) \cap \text{Ker}(Z_2) = \{ z_{p+1} \}$. But since $r_1 r_2 \cdots r_v$ represents a 'directed path' in the diagram with $r_1 = 1$, the path $r_1 \cdots r_v$ must have maximal length, and thus so too does $j_1 \cdots j_v$, in which case $j_v = z_{p+1}$. This contradicts the assumption $i \not= 1$ and thus shows that $(i,1) \not\in I_0^*$. This is in fact the only place where we needed the assumption that $z_{p+1}$ is the sink. So $\{ (i,1) : 1 \leq i \leq n \}$ represent distinct conjugacy classes.

It remains to show that if $(i,s) \in I \setminus I_0^*$ then $(i,s) \sim (k,1)$ for some $k$. If $s = 1$ then we are done, and so we assume $2 \leq s < p+1$. Since $z_s$ is not a source, there is $Z \in \{X,Y\}$ with $Z(z_{s-1}^*) = z_s^*$. If $z_i^* \not\in \{Z(z_k^*)\}$ then $Z(z_i^*) \not= 0$ and $Z(z_s^*) = 0$; a contradiction. So $z_i^* = Z(z_j^*)$ for some $z_j^*$, and hence $(i,s) \sim (j,s-1)$ with $(j,s-1) \not\in I_0^*$. By induction, $(j,s-1) \sim (k,1)$ for some $k$. A similar inductive argument holds if $p+1 \leq s \leq n$.

Returning to $\mathbb{E}$, if for $T \in M_n(k)$ we write $T = \sum t_{is} E_{is}$, then (\ref{rules}) asserts $T \in \E$ iff $t_{is} = 0$ for $(i,s) \in I_0^*$ and $\mu_2 t_{is} = \mu_1 t_{jr}$ whenever $Z : (i,s) \rightarrow (j,r)$. By the second rule we obtain $t_{11} = t_{ii}$ for $1 \leq i \leq n$. So if $i \not=1$, then by the fact that $\mathcal{R}$ linearly orders each equivalence class from $I \setminus I_0^*$, we see that $t_{i1}$ uniquely determines $t_{jr}$ for $(j,r) \in [i,1]$. More precisely, we obtain basis elements $T_i$ of $\E$ for $1 \leq i \leq n$ by setting $t_{j1} = \delta_{ij}$. Note that $T_1 = 1_{\E}$. Since $\mathbb{E}$ is a local algebra, $J(\E)$ consists of all non-units in $\mathbb{E}$. If $T_i$ is a unit, then write $T_i = \sum t_{jk} E_{jk}$ and note that $t_{1k} \not= 0$ for some $k$. However, $(1,k) \in I_0$ for $k \not= 1$, and so $k = 1$. Hence, $(1,1) \sim (i,1)$ so that $T_i = T_1$. Thus, $\{ T_i : 2 \leq i \leq n \}$ is a basis of $J(\E)$.

We now consider multiplication of $\{ T_i \}$. Let $T_{k_1}, T_{k_2} \in J(\E)$ and suppose $T_{k_2}T_{k_1} = \sum t_k T_k$ for some scalars $t_k$. If $k$ is such that $t_k \not= 0$, then $E_{k1} = E_{ij}E_{lm}$ for some $(i,j) \in [k_2,1]$ and $(l,m) \in [k_1,1]$. In particular, $i = k$, $j = l$, and $m = 1$. So $l = k_1$ and $(k_2,1) \sim (k,k_1)$. Since $k_2 \not= 1$ we know that there is a maximal chain with $(k_2,1)$ as its left-most element, and from the chain we obtain two directed paths in the diagram. In particular, since $(k,k_1)$ appears in this chain with $k_1$ fixed, there is at most one $k$ for which $t_k \not= 0$. In fact, if $T_{k_2}(T_{k_1}(z_1)) = \mu z_k$ for some $\mu \not= 0$ then $T_{k_2}T_{k_1} = \mu T_k$, and if $T_{k_2}(T_{k_1}(z_1)) = 0$ then $T_{k_2}T_{k_1} = 0$. Note that $(2,1)$ and $(n,1)$ have maximal chains given by the following:

$$(2,1) \rightarrow (3,n) \rightarrow (4,n-1) \rightarrow \cdots \rightarrow (p+1,p+2)$$
$$(n,1) \rightarrow (n-1,2) \rightarrow (n-2,3) \rightarrow \cdots \rightarrow (p+1,p)$$

where we have suppressed $\mathcal{R}$ in favor of $\rightarrow$. In particular, the sequence $T_2(z_1^*)$, $T_nT_2(z_1^*)$, $T_2T_nT_2(z_1^*)$, etc. is given by $z_2^*$, $z_{n-1}^*$, $z_4^*$, $z_{n-3}^*, \ldots, z_{p+1}^*,0$, and the $T_n(z_1^*)$, $T_2T_n(z_1^*)$, $T_nT_2T_n(z_1^*)$, etc. is given by $z_n^*$, $z_3^*$, $z_{n-2}^*$, $z_5^*$, $\ldots, z_{p+1}^*,0$. This shows that the unitary subalgebra of $\mathbb{E}$ generated by $T_2$ and $T_n$ equals $\mathbb{E}$. Moreover by (\ref{rules}) we have

\begin{equation}\label{T2TN}
\begin{split}
T_2 &= E_{21} + \lambda^{-1}E_{3n} + \lambda^{-1}E_{4,n-1} + \cdots + \lambda^{-1}E_{p+1,p+2} \\
T_n &= E_{n1} + E_{n-1,2} + E_{n-2,3} + \cdots + E_{p+1,p}
\end{split}
\end{equation}

It is convenient to introduce some notation: given an algebra $A$ and $x,y \in A$ define $(xy)_i = xyx \cdots$ where we take the product of $i$ many elements. Since $T_2(T_2(z_1^*)) = T_2(z_2^*) = 0$ and $T_n(T_n(z_1^*)) = 0$, we obtain $T_2^2 = T_n^2 = 0$, and also $(T_2T_n)_p  = \mu(T_n T_2)_p$ for some $0 \not= \mu \in k$. In fact, we can use (\ref{T2TN}) to check that $\mu = 1$ if $n \equiv_4 2$ and $\mu = \lambda^{-1}$ if $n \equiv_4 0$. In particular, $\E$ has the basis

$$\{ T_1, T_2, T_n, T_2T_n, T_nT_2, \ldots, (T_2T_n)_{p-1}, (T_nT_2)_{p-1}, (T_2T_n)_p \}$$

Define $\eta : \E \rightarrow k$ by sending the first $n-1$ basis elements to zero and $\eta((T_2T_n)_p) = 1$. From $T_2^2 = T_n^2$ it is easy to show that $\text{Ker}(\eta)$ contains no nonzero left or right ideals, so that $\E$ is a Frobenius algebra and hence weakly symmetric (as per remarks in section 2). If $n \equiv_4 2$ (and $n > 2$) then $p$ is odd and $[T_2,(T_nT_2)_{p-1}] = (T_2T_n)_p$. So any symmetrizing form $\eta' : \E \rightarrow k$ vanishes on $(T_2T_n)_p$, which is a contradiction since $(T_2T_n)_p$ generates a 1-dimensional ideal in $\E$. Thus, $\E$ is not symmetric. The same is true if $n \equiv_4 0$ and $\lambda \not= 1$. On the other hand, suppose $n \equiv_4 0$ and $\lambda = 1$, and note that there is a grading of $\E$ obtained by assigning $T_2$ and $T_n$ the weight one. By definition, $\eta$ vanishes on homogeneous elements with weight unequal to $p$. Since $[\E,\E]$ is a graded subspace of $\E$ and the only commutator with weight $p$ equals $(T_2T_n)_p-(T_nT_2)_p = 0$, we see that $\eta$ is a symmetrizing form for $\E$. The proof is complete.
\end{proof}

We now turn briefly to the consideration of the indecomposable $kD_{4q}$-modules. These fall into one of three types: the left regular module; string modules $M_w$ where $w$ does not contain $(ab)^q$, $(ba)^q$, or their inverses as subwords; and $M(w,m,\lambda)$ where no power of $w$ contains $(ab)^q$, $(ba)^q$, or their inverses as subwords. Since $E_{D_{4q}}(M) \simeq E_{\Lambda}(M)$ in the second and third cases, we may make use of Theorems \ref{StringModules} and \ref{BandModules}. For instance, if $q = 1$ then the only indecomposable $k(\mathbb{Z}_2 \times \mathbb{Z}_2)$-modules with symmetric endomorphism algebras are $k$, $M_a$, $M_b$, $M(ab^{-1},1,\lambda)$ for $0 \not= \lambda \in k$, and the left regular module. More generally, we see that $kD_{4q}$ has infinitely many indecomposable modules with symmetric endomorphism algebra, a result that will be established in greater generality in section 5. The consideration of when $E_{D_{4q}}(M)$ is symmetric for $M$ an arbitrary $kD_{4q}$-module will be postponed until section 5, where we will consider more generally the case of local algebras.

\section{Nakayama Algebras and Uniserial Modules}

In this section we extend the analysis carried out for cyclic groups to a larger class of algebras known as Nakayama algebras. Recall that $\Lambda$ is Nakayama if its left and right regular modules are direct sums of uniserial modules. In Theorem \ref{Nakayama} we shall classify the $\Lambda$-modules whose endomorphism algebra is symmetric. In fact, several of our methods are applicable to uniserial modules for an arbitrary algebra, and we shall begin with these, after first providing a useful lemma.

\begin{lemma}\label{notSym}Suppose $\Lambda$ is an algebra with modules $M_1$ and $M_2$. If there is $0 \not= \beta \in \text{Hom}_{\Lambda}(M_2,M_1)$ such that $\beta^*(\text{Hom}_{\Lambda}(M_1,M_2)) = 0$ then $\mathbb{E} = E_{\Lambda}(M_1 \oplus M_2)$ is not symmetric.\end{lemma}
\begin{proof}
Recall that we can write

\begin{equation}\label{decomp}
\E \simeq
\begin{pmatrix}
E_{\Lambda}(M_1) & \text{Hom}_{\Lambda}(M_2,M_1) \\
\text{Hom}_{\Lambda}(M_1,M_2) & E_{\Lambda}(M_2)
\end{pmatrix}
\end{equation}

Let $\lambda : \mathbb{E} \rightarrow k$ be a linear map that vanishes on commutators. So if $i \not= j$ and $\alpha \in \text{Hom}_{\Lambda}(M_j,M_i)$ then

$$0 = \lambda((\alpha E_{ij})E_{jj} - E_{jj}(\alpha E_{ij})) = \lambda(\alpha E_{ij})$$

We claim that $\mathbb{E}(\beta E_{12}) \subseteq \text{Ker}(\lambda)$. For this, note that

$$\lambda\left(\left(\sum \alpha_{ij} E_{ij}\right)(\beta E_{12})\right) = \lambda(\alpha_{11}\beta E_{12}) + \lambda(\alpha_{21}\beta E_{22}) = 0$$

since $\alpha_{21}\beta = 0$ for all $\alpha \in \text{Hom}_{\Lambda}(M_1,M_2)$. Therefore, $\lambda$ is not a symmetrizing form and $\E$ is not symmetric.
\end{proof}

We now parameterize the Hom space between two uniserial modules.

\begin{prop}\label{uniserial}Suppose $M_1, M_2 \in {_{\Lambda}\text{mod}}$ are uniserial and let

$$\mathcal{S}(M_1,M_2) = \{ 1 \leq l \leq \min\{ \ell\ell(M_1), \ell\ell(M_2)\} : M_1/J^l(M_1) \simeq \text{Soc}^l(M_2) \}$$

For each $l \in \mathcal{S}(M_1,M_2)$ fix an isomorphism $M_1/J^l(M_1) \simeq \text{Soc}^l(M_2)$ and write $\alpha_l$ for the composition

$$M_1 \twoheadrightarrow M_1/J^l(M_1) \simeq \text{Soc}^l(M_2) \hookrightarrow M_2$$

Then $\{ \alpha_l : l \in \mathcal{S}(M_1,M_2) \}$ is a basis of $\text{Hom}_{\Lambda}(M_1,M_2)$.\end{prop}
\begin{proof}
Since $k$ is algebraically closed, the result is immediate if $M_1$ or $M_2$ is simple. If $M_2$ is not a quotient of $M_1$ then every homomorphism $M_1 \rightarrow M_2$ has image contained in $J(M_2)$, and so the result follows by induction on $\ell\ell(M_2)$. So suppose $M_2$ is a quotient of $M_1$ and write $M_2 = M_1/\text{Soc}^l(M_1)$. Every homomorphism factors through $M_1 \twoheadrightarrow M_1/\text{Soc}^l(M_1)$, and so the result follows by induction on $\ell\ell(M_1)$ if $l > 0$. So we may suppose $M_1 = M_2 = M$. Note that if $L = \ell\ell(M)$ then

$$\mathcal{S}(M,M) = \{L\} \coprod \mathcal{S}(M,J(M))$$

There is $m \in M$ for which $M/J(M) = {\Lambda}(m+J(M))$. So if $\psi \in \mathbb{E}_{\Lambda}(M)$, then there is $\lambda_1 \in k$ for which $\psi(m) = \lambda_1 m + m'$ for some $m' \in J(M)$. Also, $\alpha_L(m) = \lambda_2 m + m''$ for some $0 \not= \lambda_2 \in k$ and $m'' \in J(M)$. Hence, $\psi - (\lambda_1/\lambda_2)\alpha_L$ is a homomorphism $M \rightarrow J(M)$. By induction, $\psi-(\lambda_1/\lambda_2)\alpha_L = \sum c_l\alpha_l$ for some $c_l \in k$. This shows that $\{ \alpha_l \}_{l \in \mathcal{S}(M,M)}$ spans $E_{\Lambda}(M)$. For linear independence, it is clear that $\alpha_L$ is not a linear combination of $\{ \alpha_l \}_{l \in \mathcal{S}(M,J(M))}$ and the linear independence of $\{ \alpha_l \}_{l \in \mathcal{S}(M,J(M))}$ follows by induction. The proof is complete.
\end{proof}

This parametrization leads to a criterion for $E_{\Lambda}(M)$ to be symmetric when $M$ is uniserial.

\begin{thm}\label{uniEnd}Suppose $M$ is a uniserial $\Lambda$-module. Then $\mathbb{E} = E_{\Lambda}(M)$ is quasi-Frobenius iff $\mathcal{S}(M,M) = \{ \ell\ell(M)\}$ or if there is $d_1 \in \mathbb{N}^+$ with

$$\mathcal{S}(M,M) = \{ \ell\ell(M) - id_1 : 0 \leq i \leq \lceil \ell\ell(M)/d_1-1 \rceil \}$$

Moreover, $\mathbb{E}$ is symmetric whenever $\mathbb{E}$ is quasi-Frobenius.
\end{thm}
\begin{proof}
Let $d$ be the largest integer less than $\ell\ell(M)$ contained in $\mathcal{S}(M,M)$; if no such $d$ exists, then $\mathbb{E} = k\text{Id}_M$ by Proposition \ref{uniserial} and hence $\mathbb{E}$ is symmetric. Write $F_1 = \alpha_d$ and $d' = \ell\ell(M)-d$ so that $F_1(M) = J^{d'}(M)$. Since $(F_1)^i$ maps $M$ onto $J^{d'i}(M)$ and $\ell\ell(J^{d'i}(M)) = \ell\ell(M)-d'i$ provided $d'i \leq \ell\ell(M)$, we see that $\ell\ell(M)-d'i \in \mathcal{S}(M,M)$ for $0 \leq d'i < \ell\ell(M)$. In particular, $\{ \ell\ell(M) - id' : 0 \leq i \leq i^* \} \subseteq \mathcal{S}(M,M)$ where $i^* = \lceil \ell\ell(M)/d'-1 \rceil$ and we may assume $\alpha_{\ell\ell(M)-id'} = (F_1)^i$.

If $\mathcal{S}(M,M)$ has the form given in the statement for some $d_1$, then $d_1 = d'$ by maximality of $d$, and so $\mathbb{E}$ has basis $\{ \text{Id}_M \} \coprod \{ (F_1)^i : 1 \leq i \leq i^* \}$ so that $\mathbb{E} \simeq k[T]/(T^{i^*+1})$ is symmetric. Suppose then that there is $j \in \mathcal{S}(M,M)$ with $d' \nmid (\ell\ell(M)-j)$, and let $\alpha_j : M \rightarrow M$ be a map with image $\text{Soc}^j(M) = J^{j'}(M)$ where $j' = \ell\ell(M)-j$. Note that $(F_1)^i\alpha_j$ maps $M$ onto $J^{j'+id'}$ and so $\ell\ell(M)-(j'+id') \in \mathcal{S}(M,M)$ provided $j'+id' < \ell\ell(M)$. Choose $i_1 \in \mathbb{N}$ maximal subject to this condition. If $\beta$ is a non-automorphism of $M$, then $\beta(M) \subseteq J^{d'}(M)$ by choice of $d$. So $(F_1)^{i_1}\alpha_j \not= 0$ and $\beta(F_1)^{i_1}\alpha_j$ maps $M$ into $J^{j'+(i_1+1)d'}(M) = 0$ since $j'+(i_1+1)d' \geq \ell\ell(M)$ by choice of $i_1$. In particular, $(F_1)^{i_1}\alpha_j$ is a nonzero element of $\text{Soc}({_{\mathbb{E}}\mathbb{E}})$. A similar argument shows that $(F_1)^{i_2} \in \text{Soc}({_{\mathbb{E}}\mathbb{E}})$ for $i_2$ chosen maximal subject to $d'i < \ell\ell(M)$. As $j'+i_1d' \not= i_2d'$ for any $i_1, i_2$, we conclude that $\dim \text{Soc}({_{\mathbb{E}}\mathbb{E}}) \geq 2$ and hence $\mathbb{E}$ is not quasi-Frobenius. The result is established.
\end{proof}

Note that it is not true that $E_{\Lambda}(M)$ is symmetric for every uniserial module $M$ - take $M$ to be the string module $M_{aba}$. This fact is true, however, when we suppose that $\Lambda$ is Nakayama. Recall from \cite{Elements} that $\Lambda$ is Nakayama provided each block $\Gamma$ of the basic algebra associated with $\Lambda$ has an ext quiver that has one of the following two forms:

\begin{figure}[here]
\centerline{
\xymatrix{\circ & \circ \ar[l] & \circ \ar[l] & \cdots \ar[l] & \circ \ar[l] & \circ \ar[l]}}
\caption{$\Gamma$ has a simple projective}
\label{fig1}
\end{figure}

\begin{figure}[here]
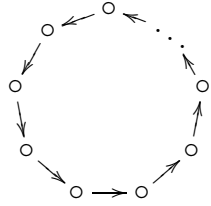

\centerline{
\def\alphanum{\ifcase\xypolynode\or \circ \or \ddots \or \circ \or \circ \or \circ \or \circ \or \circ \or \circ \or \circ \or \circ \or \circ \or \circ \or \circ\fi}
\xy/r3pc/:
{\xypolygon9{~><{}
~*{\alphanum}
~>>{}}}
\endxy}
\caption{$\Gamma$ has no simple projective}
\label{fig1}
\end{figure}

There is no harm in assuming that $\Lambda$ is a basic and connected Nakayama algebra, and so we make this assumption throughout the rest of this section. If $\text{Irr}(\Lambda) = \{S_1,\ldots,S_n\}$ with $P_i$ a projective cover of $S_i$, then every indecomposable $\Lambda$-module $M$ can be uniquely written as $M = P_i/J^j(P_i)$ for some $1 \leq i \leq n$ and $1 \leq j \leq \ell\ell(P_i)$. In particular, $M$ is specified by $\text{Top}(M)$ and $\ell\ell(M)$. Knowing this information and $n$, we may write down the composition factors of $M$. For example, if $n = 3$, $\text{Top}(M) = S_2$, and $\ell\ell(M) = 7$ then $M$ has composition factors $S_2, S_3, S_1, S_2, S_3, S_1, S_2$. Note that the quiver of $\Lambda$ must be of the second form in this case, and observe the periodicity that is displayed by the composition factors. Moreover, an isomorphism of the form $M_1/J^l(M_1) \simeq \text{Soc}^l(M_2)$ arises for $1 \leq l \leq \min\{ \ell\ell(M_1), \ell\ell(M_2) \}$ precisely when $\text{Top}(M_1) \simeq \text{Top}(\text{Soc}^l(M_2))$. Therefore, $\mathcal{S}(M_1,M_2)$ can be readily determined from knowledge of $M_1$ and $M_2$, and Proposition \ref{uniserial} provides a 'combinatorial' parametrization of $\text{Hom}_{\Lambda}(M_1,M_2)$. In fact, Theorem \ref{uniEnd} specializes to the following.

\begin{Cor}\label{NakCor}Suppose $\Lambda$ is a basic and connected Nakayama algebra with $M \in \text{Ind}(\Lambda)$. Then $\E = E_{\Lambda}(M)$ is symmetric.\end{Cor}
\begin{proof}
Write $S = \text{Top}(M)$, $n = |\text{Irr}(\Lambda)|$, and suppose $S$ has multiplicity $m \geq 1$ in $M$. By the previous paragraph, we see that $\mathcal{S}(M,M) = \{ \ell\ell(M)-in : 0 \leq i \leq m-1\}$ and so $\E$ is symmetric by Theorem \ref{uniEnd}.
\end{proof}

It is convenient to introduce some notation: for $M \in \text{Ind}(\Lambda)$ and $S \in \text{Irr}(\Lambda)$ we write $m(M,S)$ for the multiplicity of $S$ in $M$, that is, the number of times $S$ occurs as a composition factor of $M$. We now classify the $\Lambda$-modules with symmetric endomorphism algebra.

\begin{thm}\label{Nakayama}Suppose $\Lambda$ is a basic and connected Nakayama algebra. Let $M \in {_{\Lambda}\text{mod}}$ with non-isomorphic indecomposable direct summands $M_1, \ldots, M_r$. Then $E_{\Lambda}(M)$ is symmetric iff for $i \not= j$ the modules $M_i$ and $M_j$ satisfy one of the following:

\begin{enumerate}

\item[a.] $\text{Top}(M_2)$ has multiplicity zero in $\text{Soc}^l(M_1)$ and $\text{Top}(M_1)$ has multiplicity zero in $\text{Soc}^l(M_2)$ where $l = \min\{ \ell\ell(M_1),\ell\ell(M_2) \}$.

\item[b.] For some $m \geq 1$, $\text{Top}(M_1)$ has multiplicity $m+1$ in $M_1$ and $m$ in $M_2$, and $\text{Top}(M_2)$ has multiplicity $m+1$ in $M_2$ and $m$ in $M_1$.

\end{enumerate}

In particular, $\text{Top}(M_i) \not\simeq \text{Top}(M_j)$ if $i \not= j$.
\end{thm}
\begin{proof}
Assume that $E_{\Lambda}(M)$ is symmetric so that $\E := E_{\Lambda}(M_1 \oplus M_2)$ is symmetric. It suffices to show a. or b. holds for $M_1$ and $M_2$. We may assume $\ell\ell(M_2) \geq \ell\ell(M_1)$. Notice that a. is equivalent to $\text{Hom}_{\Lambda}(M_1,M_2) = \text{Hom}_{\Lambda}(M_2,M_1) = 0$. So suppose a. does not hold and let $S_{n_i} = \text{Top}(M_i)$. We must have $\text{Hom}_{\Lambda}(M_1,M_2)$ and $\text{Hom}_{\Lambda}(M_2,M_1)$ are both nonzero by Lemma \ref{notSym}. So we may define $N_1$ as the largest proper submodule of $M_1$ with $\text{Top}(N_1) \simeq S_{n_2}$, and $P_1$ as the smallest submodule of $M_1$ with $\text{Top}(P_1) \simeq S_{n_2}$; similarly we may define $N_2$ and $P_2$ with $\text{Top}(N_2) \simeq \text{Top}(P_2) \simeq S_{n_1}$. Also, write $\alpha_1$ and $\alpha_2$ for the maps $M_1 \twoheadrightarrow P_2$ and $M_2 \twoheadrightarrow P_1$, respectively. We aim to show that $m = m(M_1,S_{n_1}) - 1$ satisfies the conditions of b. This breaks into two cases. Note first that

\begin{equation}\label{mult}
\begin{split}
m(M_1,S_{n_2}) &\leq m(M_1,S_{n_1}) \leq m(M_1,S_{n_2})+1 \\
m(M_2,S_{n_1}) &\leq m(M_2,S_{n_2}) \leq m(M_2,S_{n_1})+1
\end{split}
\end{equation}

Assume that $S_{n_1} \not\simeq S_{n_2}$. If $\beta \in \text{Hom}_{\Lambda}(M_2,M_1)$ then $\text{Im}(\beta) \subseteq N_1 \subset M_1$ and so $m(\text{Im}(\beta),S_{n_1}) \leq m(N_1,S_{n_1}) < m(M_1,S_{n_1})$. In particular, if $m(M_2,S_{n_1}) \geq m(M_1,S_{n_1})$ then $P_2 \subseteq \text{Ker}(\beta)$ and so $\beta\alpha_1 = 0$ for all $\beta \in \text{Hom}_{\Lambda}(M_2,M_1)$, rendering a contradiction by Lemma \ref{notSym}. Thus, $m(M_2,S_{n_1}) < m(M_1,S_{n_1})$ and similarly $m(M_1,S_{n_2}) < m(M_2,S_{n_2})$. Since $\ell\ell(M_2) \geq \ell\ell(M_1)$, it follows that $m(M_2,S_{n_2}) \geq m(M_1,S_{n_1})$ and hence $m(M_2,S_{n_1}) \geq m$ by (\ref{mult}). Thus, $m(M_2,S_{n_1}) = m$, $m(M_2,S_{n_2}) = m+1$, and similarly $m(M_1,S_{n_2}) = m$.

On the other hand, if we assume $S_{n_1} = S_{n_2} = S$ then $\ell\ell(M_1) < \ell\ell(M_2)$ since $M_1 \not\simeq M_2$. In particular, $\text{Im}(\beta) \subseteq N_2$ whenever $\beta : M_1 \rightarrow M_2$. If $m(N_2,S) < m(M_1,S)$ then $\beta\alpha_2 = 0$ for all $\beta : M_1 \rightarrow M_2$, rendering a contradiction. On the other hand, if $m(N_2,S) \geq m(M_1,S)$ then $\beta\alpha_1 = 0$ for all $\beta : M_2 \rightarrow M_1$ since $m(N_2,S) < m(M_2,S)$, rendering the final contradiction. So $S_{n_1} = S_{n_2}$ never occurs. In fact, it is clear that $\text{Top}(M_1) \not\simeq \text{Top}(M_2)$ in both cases a. and b., thus completing the necessity of a. and b.

Suppose then that $M = \oplus_{i=1}^r M_i^{\oplus m_i}$ with $M_i$ and $M_j$ non-isomorphic indecomposable modules satisfying a. or b. whenever $i \not= j$. Also recall that $n = |\text{Irr}(\Lambda)|$. Set $\E = E_{\Lambda}(M)$ and note that if $e_{ij}$ denotes projection onto the $j$th copy of $M_i$ in $M$, then $1 = \sum_{ij} e_{ij}$ is the Pierce decomposition of $1$ in $\E$, so that $e\E e$ is a basic algebra for $e = \sum e_{i1}$. Since $e\E e \simeq E_{\Lambda}(\oplus M_i)$ and symmetry is preserved under Morita equivalence, we may assume $m_i = 1$ for all $i$. Let $T$ consist of all $i$ such that $(M_i,M_j)$ satisfies b. for some $j \not= i$. If $i_1, i_2 \in T$ then $M_{i_1}/J^n(M_{i_1})$ has composition factors $S_u, S_{u+1}, \ldots, S_{u+n-1}$ and $\text{Soc}^n(M_{i_2})$ has composition factors $S_v, S_{v+1}, \ldots, S_{v+n-1}$ for some $u,v$, where $S_j$ is considered for $j$ modulo $n$, and hence $\text{Hom}_{\Lambda}(M_{i_1},M_{i_2}) \not= 0$. That is, $(M_{i_1},M_{i_2})$ satisfies b. Hence, upon relabeling the $\{ M_i \}$ if necessary, there is $1 \leq s \leq r$ such that $\text{Hom}_{\Lambda}(M_i,M_j) = \text{Hom}_{\Lambda}(M_j,M_i) = 0$ for $1 \leq i < s$ and $j \not= i$, and $(M_i,M_j)$ satisfy b. for $i,j \geq s$ and $j \not= i$. In particular

$$\E \simeq \prod_{i=1}^{s-1} E_{\Lambda}(M_i) \times E_{\Lambda}(M_s \oplus \cdots \oplus M_r)$$

and so we may assume $s = 1$. Moreover, if $S_{n_i} = \text{Top}(M_i)$ then we may assume $n_1 < n_2 < \cdots < n_r$. It is clear that there is $m \geq 1$ such that $S_{n_i}$ has multiplicity $m+1$ in $M_i$ for all $i$. By the proof to Corollary \ref{NakCor} we know $E_{\Lambda}(M_i) \simeq k[T]/(T^{m+1})$. More precisely, using cyclic notation for $\{M_i\}$ modulo $r$, we let $N_{ij}$ be the largest submodule of $M_i$ with $\text{Top}(N_{ij}) \simeq S_{n_j}$, $\beta_i : M_i \rightarrow M_{i-1}$ with image $N_{i-1,i}$ for $1 \leq i \leq r$, and define $\alpha_i : M_i \rightarrow M_i$ by $\alpha_i = \beta_{i+1} \cdots \beta_{i-1}\beta_i$. Then $E_{\Lambda}(M_i)$ has basis $\{1_{M_i},\alpha_i,\alpha_i^2,\ldots,\alpha_i^m\}$ and there is a symmetrizing form $\lambda_i : E_{\Lambda}(M_i) \rightarrow k$ given by $\lambda_i(\alpha_i^j) = \delta_{jm}$. Furthermore, if $i \not= j$ and $s > 0$ is the smallest number for which $i-s-1 \equiv_r j$, then we let $\beta_{ji} : M_i \rightarrow M_j$ by $\beta_{ji} = \beta_{i-s} \cdots \beta_{i-1}\beta_i$. Then by the description in b. we see that $\text{Hom}_{\Lambda}(M_i,M_j)$ has basis $\{ \alpha_j^u\beta_{ji} : 0 \leq u \leq m-1\}$ where we interpret $\alpha_j^0 = 1_{M_j}$.

Now define $\lambda : \E \rightarrow k$ by extending linearly the rule $\lambda(\gamma_{ij}E_{ij}) = \delta_{ij}\lambda_i(\gamma_{ij})$ for $\gamma_{ij} : M_j \rightarrow M_i$. Suppose that $\zeta = \sum \gamma_{ij} E_{ij}$ satisfies $\zeta\E \subseteq \text{Ker}(\lambda)$. If $\beta : M_i \rightarrow M_j$ then $0 = \lambda(\zeta \cdot \beta E_{ji} \cdot \alpha E_{ii}) = \lambda_i(\gamma_{ij}\beta\alpha)$ for all $\alpha \in E_{\Lambda}(M_i)$, and hence $\gamma_{ij}\beta = 0$. If $i = j$ then $\gamma_{ij} = 0$ since we may take $\beta = 1_{M_i}$, and so we assume $i \not= j$, in which case $\gamma_{ij}\beta_{ji} = 0$ so that $N_{ji} \subseteq \text{Ker}(\gamma_{ij})$. In particular, if $\gamma_{ij} \not= 0$ then $m(\text{Im}(\gamma_{ij}),\text{Top}(M_j)) = 1$ and $m(\text{Im}(\gamma_{ij}),\text{Top}(M_i)) = 0$. Since $\text{Im}(\gamma_{ij})$ is a submodule of $M_i$, this contradicts the description provided in b. So $\zeta = 0$ and it remains to check that $\lambda$ vanishes on commutators. If $\gamma : M_i \rightarrow M_j$ and $\gamma' : M_k \rightarrow M_l$ then

$$[\gamma E_{ij}, \gamma' E_{kl}] = \delta_{jk} \gamma\gamma' E_{il} - \delta_{li} \gamma'\gamma E_{kj}$$

and $\lambda$ maps this element to zero, except possibly when $i = l$ and $j = k$, in which case it is sent to $\lambda_i(\gamma\gamma') - \lambda_j(\gamma'\gamma)$. If $i = j$ then $\lambda_i(\gamma\gamma') = \lambda_j(\gamma'\gamma)$ since $E_{\Lambda}(M_i)$ is commutative, and if $i \not= j$ then we may suppose $\gamma = \alpha_j^u\beta_{ji}$ and $\gamma' = \alpha_i^v\beta_{ij}$. From the definitions we see $\alpha_i\beta_{ij} = \beta_{ij}\alpha_j$ and so $\gamma\gamma' = \beta_{ji}\beta_{ij}\alpha_j^{u+v}$ and $\gamma'\gamma = \beta_{ij}\beta_{ji}\alpha_i^{u+v}$. We may also check that $\beta_{ij}\beta_{ji} = \alpha_i$ and $\beta_{ji}\beta_{ij} = \alpha_j$. It follows from the definition of $\lambda_i$ and $\lambda_j$ that $\lambda(\gamma\gamma') = \lambda(\gamma'\gamma)$, and hence $\lambda$ is a symmetrizing form for $\E$.

\end{proof}

For the details in the second half of the previous proof, it is constructive to consider an example:

$$
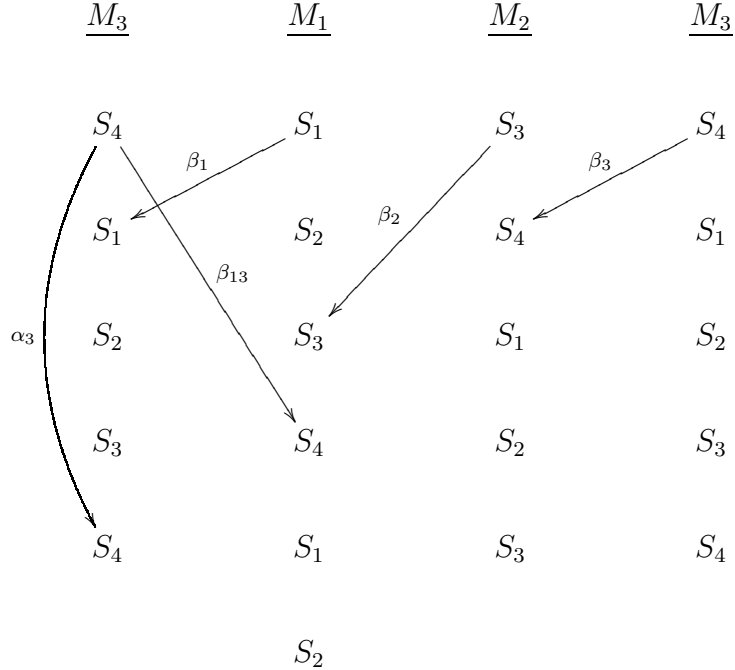
\xymatrix{\underline{M_3} & & \underline{M_1} & & \underline{M_2} & & \underline{M_3} \\ S_4 \ar[dddrr]^{\beta_{13}} \ar@/_2pc/[dddd]_{\alpha_3} && S_1 \ar[dll]_{\beta_1} && S_3 \ar[ddll]_{\beta_2} && S_4 \ar[dll]_{\beta_3}\\ S_1 && S_2 && S_4 && S_1 \\ S_2 && S_3 && S_1 && S_2 \\ S_3 && S_4 && S_2 && S_3 \\ S_4 && S_1 && S_3 && S_4 \\ && S_2 && &&}$$

Here the uniserial $M_i$ are specified by their composition series, which is possible since $\Lambda$ is Nakayama, and the $\beta_i$ are specified by their images, though $\beta_i$ are not uniquely determined. Also, $\alpha_3$ and $\beta_{13}$, for instance, are defined as appropriate compositions of $\{ \beta_1, \beta_2, \beta_3 \}$. It is because we have carefully chosen bases for $E_{\Lambda}(M_i)$ and $\text{Hom}_{\Lambda}(M_i,M_j)$ that we can assert $\lambda$ vanishes on commutators.

\section{Local Algebras and Future Research}

In this final section we complete the analysis started in section 3 and tie this to the results obtained for local Nakayama algebras. More generally, for a local algebra $\Lambda$, we show that the problem of determining when $E_{\Lambda}(M)$ is symmetric or quasi-Frobenius reduces to the consideration of the indecomposable $\Lambda$-modules. After this, we provide some further results that are of interest in their own right and point towards possible future research. To begin, we prove the following.

\begin{prop}\label{TopSoc}Suppose $\Lambda$ is a local algebra and $M \in \text{Ind}(\Lambda)$ is such that $\E = E_{\Lambda}(M)$ is quasi-Frobenius. Then $M/J(M) \simeq \text{Soc}(M) \simeq k$.\end{prop}
\begin{proof}
For $N$ a maximal submodule of $M$ and $Q$ a simple submodule of $M$, write $\alpha_{N,Q}$ for the composition

$$M \twoheadrightarrow M/N \stackrel{\sim}{\rightarrow} Q \hookrightarrow M$$

where $\sim$ denotes an arbitrary automorphism. Let $0 \not= \alpha \in \text{Soc}({_{\mathbb{E}}\mathbb{E}})$ so that $\beta\alpha = 0$ and hence $\text{Im}(\alpha) \subseteq \text{Ker}(\beta)$ whenever $\beta$ is a non-automorphism of $M$. Let $R$ be a simple submodule of $\text{Im}(\alpha)$ and note that $\beta\alpha_{N,R} = 0$ for $\beta \in J(\mathbb{E})$ and $N$ a maximal submodule. This means $\alpha_{N,R} \in \text{Soc}({_{\mathbb{E}}\mathbb{E}})$. Therefore, since $\E$ is quasi-Frobenius, $M$ has a unique maximal submodule, equal to $J(M)$, and hence $M/J(M) \simeq k$. Moreover, for $Q$ a simple submodule of $\text{Soc}(M)$, we see $\alpha_{J(M),Q}\beta = 0$ whenever $\beta \in J(\mathbb{E})$. So $\alpha_{J(M),Q} \in \text{Soc}(\mathbb{E}_{\mathbb{E}})$ and hence the quasi-Frobenius condition implies that $\text{Soc}(M) \simeq k$, as required.
\end{proof}

We can now deliver the promised result.

\begin{thm}\label{localAlgebra}Suppose $\Lambda$ is a local algebra and $M \in {_{\Lambda}\text{mod}}$. Then $E_{\Lambda}(M)$ is quasi-Frobenius precisely when $M$ is isotypic, say $M = N^{\oplus e}$ for some $e \geq 1$ and $N \in \text{Ind}(\Lambda)$, with $E_{\Lambda}(N)$ quasi-Frobenius.\end{thm}
\begin{proof}
If $M = N^{\oplus e}$ with $E_{\Lambda}(M)$ quasi-Frobenius, then $E_{\Lambda}(M) \simeq M_n(E_{\Lambda}(N))$ is quasi-Frobenius. So assume $M \in {_{\Lambda}\text{mod}}$ with $E_{\Lambda}(M)$ quasi-Frobenius and suppose $M$ has two non-isomorphic indecomposable direct summands $M_1$ and $M_2$. Then $\E = E_{\Lambda}(M_1 \oplus M_2) = eE_{\Lambda}(M)e$ is quasi-Frobenius, as are $E_{\Lambda}(M_1)$ and $E_{\Lambda}(M_2)$. By Lemma \ref{TopSoc} we know $\text{Soc}(M_i) \simeq \text{Top}(M_i) \simeq k$. We can write

\begin{equation*}
\E \simeq
\begin{pmatrix}
E_{\Lambda}(M_1) & \text{Hom}_{\Lambda}(M_2,M_1) \\
\text{Hom}_{\Lambda}(M_1,M_2) & E_{\Lambda}(M_2)
\end{pmatrix}
\end{equation*}

By \cite{Drozd} we also have

\begin{equation*}
J(\E) \simeq
\begin{pmatrix}
J(E_{\Lambda}(M_1)) & \text{Hom}_{\Lambda}(M_2,M_1) \\
\text{Hom}_{\Lambda}(M_1,M_2) & J(E_{\Lambda}(M_2))
\end{pmatrix}
\end{equation*}

In terms of this matrix description, we know by the Pierce decomposition that ${_{\mathbb{E}}\mathbb{E}} = P_1 \oplus P_2$ where each $P_i = \mathbb{E}(1_{M_i}E_{ii})$ is an indecomposable projective. Since $\text{Soc}(M_i) \simeq k$, there are maps $\alpha_{ij} : M_j \rightarrow M_i$, uniquely determined up to a nonzero multiple of $k$, with image $\text{Soc}(M_i)$. We denote by $\pi$ the Nakayama permutation of $\{1,2\}$ which satisfies $\text{Soc}(P_i) \simeq \text{Top}(P_{\pi(i)})$. Also, note that $\{1_{M_1} E_{11}, 1_{M_2} E_{22} \}$ modulo $J(\mathbb{E})$ forms a basis for $\mathbb{E}/J(\mathbb{E})$, and $1_{M_i} E_{ii}$ acts on $\text{Top}(P_j)$ as the identity if $j = i$ and as zero if $j \not= i$.

Now suppose there is no inclusion $M_1 \hookrightarrow M_2$. Then every non-automorphism of $M_1$ and every homomorphism $M_1 \rightarrow M_2$ vanishes on $\text{Soc}(M_1)$. In particular, by our description of $J(\E)$ we see $s_1 = \alpha_{11} E_{11} \in \text{Soc}(P_1)$ since $M_1$ is non-simple, and also $s_2 = \alpha_{12} E_{12} \in \text{Soc}(P_2)$. Since $\text{Soc}(P_i)$ is simple, we have $\text{Soc}(P_i) = ks_i$, and since $1_{M_2} E_{22}$ annihilates $s_1$ and $s_2$, we obtain the contradiction that $\text{Soc}(P_1) \simeq \text{Soc}(P_2) \simeq \text{Top}(P_1)$. A similar contradiction arises when there is no inclusion $M_2 \hookrightarrow M_1$. Since there are inclusions $M_1 \hookrightarrow M_2$ and $M_2 \hookrightarrow M_1$ only when $M_1 \simeq M_2$, the result is established.
\end{proof}

Moreover, if $n \geq 1$, $\Gamma = M_n(\Lambda)$, $M \in {_{\Lambda}\text{mod}}$, and $M^{\oplus n}$ is the corresponding $\Gamma$-module, then $E_{\Gamma}(M^{\oplus n}) \simeq M_n(E_{\Lambda}(M))$ so that $E_{\Gamma}(M^{\oplus n})$ is symmetric or quasi-Frobenius precisely when $E_{\Lambda}(M)$ is the same. In other words, the determination of whether $E_{\Lambda}(M)$ is symmetric or quasi-Frobenius is invariant under Morita equivalences. Thus, we see that Theorem \ref{localAlgebra} holds more generally for primary algebras (i.e. $\Lambda/J(\Lambda)$ is simple). For a non-primary algebra, this result will never hold since $E_{\Lambda}(S_1 \oplus S_2) \simeq k \times k$ is symmetric whenever $S_1$ and $S_2$ are non-isomorphic irreducible $\Lambda$-modules. Proposition \ref{TopSoc} also provides a new way of thinking about indecomposable $\Lambda$-modules with symmetric endomorphism algebras, for $\Lambda$ local.

\begin{Cor}\label{interpretation}Suppose $\Lambda$ is a local algebra and $M \in \text{Ind}(\Lambda)$ has $E_{\Lambda}(M)$ symmetric. Then $M \simeq \Lambda/I$ for some left ideal $I$ where $\Gamma := \{ x \in \Lambda : Ix \subseteq I \}$ is such that $\Gamma/I$ is symmetric.\end{Cor}
\begin{proof}
Proposition \ref{TopSoc} implies that $M$ is a cyclic $\Lambda$-module, and hence has the form $M \simeq \Lambda/I$ for some left ideal $I$. An endomorphism $f$ of $\Lambda/I$ is determined by $f(1+I) = x_f + I$ and $xx_f \in I$ for all $x \in I$. So $x_f \in \Gamma$ and the map $f \mapsto x_f+I$ yields an isomorphism $E_{\Lambda}(\Lambda/I) \simeq (\Gamma/I)^{\text{op}}$.
\end{proof}

For the case of $P$-groups, this result can be used to answer in the affirmative a question raised by the computations from section 3.

\begin{Cor}Suppose $P$ is a $p$-group. Then every indecomposable $kP$-module $M$ with symmetric endomorphism algebra satisfies $\dim M \leq |P|$, and there are infinitely many such modules whenever $P$ is non-cyclic.\end{Cor}
\begin{proof}
That $\dim M \leq |P|$ is immediate from Corollary \ref{interpretation}. If the second statement is false, then let $P$ be a minimal counterexample. If $P$ is nonabelian then $\bar{P} = P/Z(P)$ is non-cyclic, and hence there are infinitely many non-isomorphic indecomposable $k\bar{P}$-modules with symmetric endomorphism algebras. Application of the inflation functor ${_{\bar{P}}\text{mod}} \rightarrow {_P\text{mod}}$ yields a contradiction. Therefore, $P$ is abelian. For $0 \not= z \in kP$ let $I_z = l(z) = r(z)$ where $l$ and $r$ denote the left and right annihilator of $(z)$ in $kP$. So Theorem 13 from \cite{NakayamaI} and Corollary \ref{interpretation} imply that $M_z = kP/I_z$ is an indecomposable $kP$-module with $kP/I_z \simeq E_P(M_z)^{\text{op}}$ symmetric. Note that $M_z \simeq M_w$ only for $(z) = (w)$ since $I_z = \text{Ann}(M_z)$. So if $\{ M_z \}$ has only finitely many isomorphism classes of $kP$-modules, then Theorem 6 from \cite{Renner} implies the existence of only finitely many ideals in $kP$. In turn, this implies that $kP$ is a principal ring, and in particular $J(kP)$ is principal. Since $kP$ is local this yields $\ell\ell(kP) = |P|$, and hence $P$ is cyclic by application of Jennings' Theorem for $p$-groups; the final contradiction.
\end{proof}

This proof generalizes to centralizer algebras $kP^Q$ with $P$ a $p$-group and $Q \leq P$, by using the results from \cite{AAA} on $\ell\ell(kP^Q)$. We already saw in Lemma \ref{notSym} one necessary condition for symmetry, when $\Lambda$ is a not necessarily local algebra. The next result is a corollary to this lemma.

\begin{Cor}Suppose $M_1$ and $M_2$ are non-isomorphic bricks (i.e. $E_{\Lambda}(M_i) \simeq k$). Then $\mathbb{E} = E_{\Lambda}(M_1 \oplus M_2)$ is symmetric iff $\text{Hom}_{\Lambda}(M_1,M_2) = \text{Hom}_{\Lambda}(M_2,M_1) = 0$.\end{Cor}
\begin{proof}
If $\text{Hom}_{\Lambda}(M_1,M_2) = \text{Hom}_{\Lambda}(M_2,M_1) = 0$ then $\mathbb{E} \simeq E_{\Lambda}(M_1) \times E_{\Lambda}(M_2) \simeq k \times k$ is semisimple and hence symmetric. Suppose either $\text{Hom}_{\Lambda}(M_1,M_2)$ or $\text{Hom}_{\Lambda}(M_2,M_1)$ is nonzero, say with $0 \not= \alpha \in \text{Hom}_{\Lambda}(M_2,M_1)$. If there is $\beta \in \text{Hom}_{\Lambda}(M_1,M_2)$ with $\beta\alpha \not= 0$, then $\beta\alpha$ is an automorphism, so that $\beta$ is surjective and hence $\dim M_1 > \dim M_2$ since $M_1 \not\simeq M_2$. Then $\gamma\beta$ is not an automorphism of $M_2$ for $\gamma \in \text{Hom}_{\Lambda}(M_2,M_1)$, and hence $\gamma\beta = 0$. So $\E$ is not symmetric by Lemma \ref{notSym}; the same holds true if $\beta\alpha = 0$ for all $\beta \in \text{Hom}_{\Lambda}(M_1,M_2)$. The proof is complete.
\end{proof}

For an example of this proposition, we could take $\Lambda$ to be a hereditary algebra with finite representation type. In general, given an algebra ${\Lambda}$, the determination of which $\Lambda$-modules have symmetric endomorphism algebras is a non-trivial problem, as demonstrated by the next proposition.

\begin{prop}An algebra $\Lambda$ is such that $E_{\Lambda}(M)$ is symmetric for all $M \in {_{\Lambda}\text{mod}}$ if and only if ${\Lambda}$ is semisimple.\end{prop}
\begin{proof}
Suppose $E_{\Lambda}(M)$ is symmetric for all $M \in {_{\Lambda}\text{mod}}$. By the remarks following Theorem \ref{localAlgebra} on Morita equivalences and the Morita invariance of semisimple algebras, we may assume that $\Lambda$ is basic. Note that ${\Lambda} \simeq E_{\Lambda}({_{\Lambda}{\Lambda}})^{\text{op}}$ is symmetric and so $\text{Soc}({\Lambda})$ is an ideal in ${\Lambda}$. If ${\Lambda}$ is not semisimple, then there is a maximal ideal $M$ of ${\Lambda}$ containing $\text{Soc}({\Lambda})$. Let $N$ be a maximal left ideal of ${\Lambda}$ containing $M$, let $S$ be a submodule of $\text{Soc}({\Lambda})$ isomorphic with ${\Lambda}/N$, and denote the inclusion $S \hookrightarrow {\Lambda}$ by $\iota$. Observe that $\text{Hom}_{\Lambda}(\Lambda,S) \simeq \text{Hom}_{\Lambda}(\Lambda/J(\Lambda),S) \simeq k$ since $\Lambda/J(\Lambda) \simeq \bigoplus_{T \in \text{Irr}(\Lambda)} T$. So if $f : {\Lambda} \rightarrow S$ is a homomorphism, then $f$ factors through the projection $\Lambda \twoheadrightarrow \Lambda/N$ and hence $S \subseteq N \subseteq \text{Ker}(f)$ so that $f\iota = 0$. We conclude by Lemma \ref{notSym} that $E_{\Lambda}(S \oplus {_{\Lambda}{\Lambda}})$ is not symmetric; a contradiction that shows that ${\Lambda}$ is semisimple. Conversely, if ${\Lambda}$ is semisimple then $E_{\Lambda}(M)$ is semisimple and hence symmetric whenever $M \in {_{\Lambda}\text{mod}}$.
\end{proof}

Lastly, as a simple example, we provide the following application of our previous results.

\begin{prop}Suppose $G$ is a $p$-nilpotent group with cyclic Sylow $p$-subgroup and $M$ a $kG$-module. Then $E_G(M)$ is symmetric if and only if $e_BM$ is an isotypic $B$-module for every block idempotent $e_B$.\end{prop}
\begin{proof}
If $B$ is a block of $kG$ then $B \simeq M_n(kD)$ for some $n \geq 1$ where $D$ is the defect group of $B$. Let $\mathcal{F}_B : {_B\text{mod}} \rightarrow {_D\text{mod}}$ be a Morita equivalence. Note that $E_G(M) \simeq \prod E_B(e_B M)$ and $E_B(e_BM)$ is symmetric precisely when $E_D(\mathcal{F}_B(e_BM))$ is symmetric. Isotopy is preserved under Morita equivalence and so the result follows by Proposition \ref{CyclicModules}.
\end{proof}

It would be interesting to see how far the methods in this paper might be pushed to analyze the symmetry of $E_G(M)$ for $G$ an arbitrary finite group and $M$ a $kG$-module. A natural starting point are the blocks with cyclic defect group. Recall that if $B$ is a block of $kG$ with cyclic defect group $D$, $Q$ the unique subgroup of $D$ with order $p$, $N = N_G(Q)$, and $b$ the unique block of $kN$ with defect group $D$ and $b^G = B$, then the Green correspondence provides a bijection between indecomposable non-projective $B$-modules and indecomposable non-projective $b$-modules. Moreover, $b$ is Nakayama with $|\text{Irr}(b)| = e$ for $e$ the inertial index of $B$ and $\ell\ell(P) = p$ whenever $P$ is an indecomposable projective $b$-module. So the results from section 4 apply to $b$. Studying the $B$-modules is now a natural way of determining how symmetry of endomorphism algebras passes through Green's correspondence.

\section{Acknowledgments}

The results from section 3 formed part of my Ph.D. thesis conducted at the University of Chicago under the direction of J. L. Alperin. I would like to thank him for his helpful supervision and constant encouragement. Theorem \ref{pNilpotent} was developed at NUI Maynooth during a research visit with John Murray, and I would like to thank the university for its generous support and John for several useful conversations.

\bibliographystyle{plain}
\bibliography{SymBib}

\end{document}